\documentclass[11pt,twoside,a4paper]{amsart}
\usepackage{amssymb}
\usepackage{amsmath,amscd}
\usepackage[initials,nobysame,shortalphabetic]{amsrefs}
\usepackage{verbatim}

\def\dbar{\bar\partial}

\def\C{{\mathbb C}}
\def\P{{\mathbb P}}

\newcommand{\Com}[1]{}

\def\be{\begin{equation}}
\def\ee{\end{equation}}

\newtheorem{thm}{Theorem}[section]
\newtheorem{lma}[thm]{Lemma}

\newtheorem{prop}[thm]{Proposition}

\theoremstyle{definition}

\theoremstyle{remark}

\newtheorem{preremark}{Remark}
\newtheorem{preex}{Example}

\newenvironment{remark}{\begin{preremark}}{\end{preremark}}

\numberwithin{equation}{section}

\ifdefined\custompagenr
\fi

\begin{document}

\title[]{An extension theorem of Ohsawa-Takegoshi type for sections of a vector bundle}


\author{Hossein Raufi}

\address{H. Raufi\\Department of Mathematics\\Chalmers University of Technology and the University of Gothenburg\\412 96 G\"OTEBORG\\SWEDEN}

\email{raufi@chalmers.se}




\begin{abstract}
Using $L^2$-methods for the $\dbar$-equation we prove that the Ohsawa-Takegoshi extension theorem also holds for holomorphic sections of a vector bundle, over compact K\"ahler manifolds. We then proceed to show that the conditions that are needed are more liberal than the ones one would need if one instead reduced the extension problem to line bundles through the usual algebraic geometric procedure of studying the projective bundle associated with the vector bundle.
\end{abstract}

\maketitle


\section{Introduction}
\noindent Let $X$ be a compact K\"ahler manifold and let $S$ be a smooth hypersurface in $X$. $S$ then defines a line bundle over $X$, which we will denote by $(S)$, which has a global holomorphic section $s$ such that $S=s^{-1}(0)$. Also let $L$ be a complex line bundle over all of $X$. The extension theorem of Ohsawa and Takegoshi, which first appeared in \cite{OT}, is a very useful theorem which has many different variants. One of the most basic forms of the theorem, the so called adjunction version, states the following. Assume that the line bundles $L$ and $(S)$ have smooth metrics $\phi$ and $\psi$ respectively, satisfying the curvature assumptions
$$i\partial\dbar\phi\geq0$$
and
$$i\partial\dbar\phi\geq\delta i\partial\dbar\psi$$
for some $\delta>0$. Assume furthermore that $s$ is normalized so that
$$|s|^2e^{-\psi}\leq e^{-1/\delta}.$$
Finally let $u$ be a global holomorphic section of $K_S+L|_S$.

Then there exists a global holomorphic section $U$ of $K_X+(S)+L$ such that
$$U=ds\wedge u$$
on $S$ and such that $U$ satsfies the estimate
$$\int_Xc_nU\wedge\bar{U}e^{-\phi-\psi}\leq C\int_Sc_{n-1}u\wedge\bar{u}e^{-\phi}$$
for some constant $C$, where we use the shorthand notation $c_p:=i^{p^2}$.

Just as in H\"ormander's $L^2$ methods approach to solving the $\dbar-$equation, much of the usefulness of the extension theorem comes from the fact that it not only gives conditions under which the extension is possible, but also provides us with an estimate for the extension. This estimate has the added merit that the constant $C$ is completely universal.

The main aim of this paper is to extend this theorem to vector bundles, i.e. to show that it is possible to replace the complex line bundle $(L,\phi)$ with the holomorphic vector bundle $(E,h)$ where $h$ is a smooth hermitian metric on $E$. The first result of this paper is the following theorem:

\begin{thm}\label{thm:1.1}
Let $X$ be a compact K\"ahler manifold and let $S$ be a smooth hypersurface in $X$, defined by a global holomorphic section $s$ of the line bundle $(S)$. Let $E$ be a holomorphic vector bundle over $X$. Assume that $E$ has a smooth hermitian metric $h$ and $(S)$ has a smooth metric $\psi$ satisfying the curvature assumptions
\be\label{eq:curv1}
i\Theta^h\geq_N0
\ee
and
\be\label{eq:curv2}
i\Theta^h-\delta i\partial\dbar\psi\otimes I\geq_N0
\ee
with $\delta>0$. Assume moreover that $s$ is normalized so that
$$|s|^2e^{-\psi}\leq e^{-1/\delta}.$$
Then for any global holomorphic section $u$ of $K_S\otimes E|_S$ there exists a global holomorphic section $U$ of $K_X\otimes(S)\otimes E$ such that
$$U=ds\wedge u$$
on $S$, and such that $U$ satisfies the estimate
$$\int_Xc_n\langle U,U\rangle_he^{-\psi}\leq C\int_Sc_{n-1}\langle u,u\rangle_h.$$
\end{thm}
Here the curvature assumptions mean that the expressions should be non-negatively curved in the sense of Nakano, and $\langle,\rangle_h$ denotes the bilinear map on bundle-valued forms associated to $h$; see section 2 where these notions are reviewed. 

There are many different ways of proving the line bundle version of the extension theorem, most of which are rather involved. Following Berndtsson \cite{B3} the main idea behind our proof is to first show that finding an extension is equivalent to solving the $\dbar$-equation
\be\label{eq:fund}
\dbar v=u\wedge[S]
\ee
where $[S]$ denotes the current of integration on $S$. Then we proceed to show that a solution to this equation exists by applying $L^2$ methods. However applying the standard existence results to (\ref{eq:fund}) will not work in this case, since the right hand side is a current and not an $L^2$-valued form, and this is where the analysis starts to get involved.

We will handle this analysis by following the approach taken in Berndtsson \cite{B1} where he arrives at the existence result and the estimate using a modified version of the $\partial\dbar$-Bochner-Kodaira method introduced by Siu in \cite{S}. In a previous paper, \cite{R}, we have shown that this method works equally well when dealing with vector bundle valued forms. Hence in this paper we prove Theorem \ref{thm:1.1} by showing that the modified method also can be adapted to vector bundles. In fact this approach works almost without change and so our presentation will follow that of \cite{B1} rather closely. This explains why the vector and line bundle versions look so similar.

After the publication of \cite{OT}, Ohsawa extended the theorem in different directions in a long series of papers. In one of these papers, \cite{O}, he obtains a result which shares some similarities to our extension theorem, although the formulation is quite different from ours, (\cite{O}, Theorem 4). We believe that our compact K\"ahler setting is slightly more general, as \cite{O} Theorem 4 only treats complex manifolds that become Stein after removing a closed subset. The main difference, however, lies in our methods of proof. We consider our adaptation of the $\partial\dbar$-Bochner-Kodaira method to the vector bundle setting to be our main originality. Furthermore, Guan and Zhou have recently proven a much more general version of the extension theorem, and also managed to determine the optimal constant in the $L^2$-estimate, (\cite{GZ}, Theorem 2.1).

For a general vector bundle $F$, a method that is widely used when one wants to generalize a result that is already known for line bundles to vector bundles, is to study the projective fiber bundle $\pi:\mathbb{P}(F)\to X$ associated to $F$, whose fiber at each point $x\in X$ is the projective space of lines in $F^*_x$. There is a naturally defined line bundle $\mathcal{O}_{\mathbb{P}(F)}(1)$ over $\mathbb{P}(F)$ which, vaguely speaking, contains all the information in $F$. Hence by studying $\mathcal{O}_{\mathbb{P}(F)}(1)\to\mathbb{P}(F)$ instead of $F\to X$, one reduces the problem back to the line bundle case. (These constructions will be reviewed in section 4.)

Now demanding that a hermitian metric is curved in the sense of Nakano is a rather strong condition, and so one may rightfully wonder what curvature assumptions this reduction procedure yields. To compare these two approaches we will first need to transform Theorem \ref{thm:1.1} to the non-adjoint case.

\begin{thm}\label{thm:1.2}
Let $F$ be a holomorphic vector bundle over a compact K\"ahler manifold $X$ and let $S$ be a smooth hypersurface in $X$, defined by a global holomorphic section $s$ of the line bundle $(S)$. Assume that $F$, $(S)$ and $K_X$ the canonical bundle of $X$ have metrics $h$, $\psi$ and $\phi_{K_X}$ respectively, satisfying the curvature assumptions
\be\label{eq:vb_curv1}
i\Theta^h-(1+\delta)i\partial\dbar\psi\otimes I-i\partial\dbar\phi_{K_X}\otimes I\geq_{N}0
\ee
and
\be\label{eq:vb_curv2}
i\Theta^h-i\partial\dbar(\psi+\phi_{K_X})\otimes I\geq_{N}0.
\ee
with $\delta>0$. Assume moreover that $s$ is normalized so that
$$|s|^2e^{-\psi}\leq e^{-1/\delta}.$$
Then any holomorphic section $U_0$ of $F$ over $S$ extends holomorphically to a section $U$ of the same bundle over $X$ satisfying the estimate
$$\int_X(U,U)_h\frac{\omega^n}{n!}\leq C\int_S(U_0,U_0)_h\frac{dS}{|ds|^2e^{-\psi}}$$
where $dS$ denotes the surface (or volume) measure on $S$ induced by the K\"ahler metric $\omega$.
\end{thm}
Just as in the line bundle case, we will see that it is not very difficult to deduce Theorem \ref{thm:1.2} once Theorem \ref{thm:1.1} has been established.

Now we can reduce the extension problem stated in Theorem \ref{thm:1.2} to line bundles by studying $\mathcal{O}_{\mathbb{P}(F)}(1)$ and $\pi^*(S)$ over $\mathbb{P}(F)$ instead. In this setting the non-adjoint version of the Ohsawa-Takegoshi theorem states that an extension is possible if the induced metrics $\pi^*h$ and $\pi^*\psi$ satisfy the curvature assumptions
\be\label{eq:lb_curv1}
i\Theta^{\mathcal{O}_{\mathbb{P}(F)}(1)}\geq(1+\delta)i\Theta^{\pi^*(S)}+i\Theta^{K_{\mathbb{P}(F)}}
\ee
and
\be\label{eq:lb_curv2}
i\Theta^{\mathcal{O}_{\mathbb{P}(F)}(1)}\geq i\Theta^{\pi^*(S)}+i\Theta^{K_{\mathbb{P}(F)}}.
\ee
The question is what conditions these imply in the original vector bundle setting. In section 4 we will prove that (\ref{eq:lb_curv1}) and (\ref{eq:lb_curv2}) imply (\ref{eq:vb_curv1}) and (\ref{eq:vb_curv2}). Hence we will see that although being curved in the sense of Nakano is a strong condition to impose on a metric, the conditions that arise when one reduces the problem to line bundles are in fact even stronger.

A key ingredient in proving these implications will be a famous theorem of Demailly and Skoda \cite{DS} which states that if a vector bundle $E$ is non-negatively curved in the sense of Griffiths, then the vector bundle $E\otimes\det E$ is non-negative in the sense of Nakano.

\section*{Acknowledgments}
\noindent It is a pleasure to thank Bo Berndtsson for inspiring and helpful discussions. I would also like to thank Mihai P{\u{a}}un for bringing the previous work by Ohsawa, \cite{O}, to my attention.

\section{The setting} 
\noindent Let $(X,\omega)$ be a compact K\"ahler manifold and let $(E,h)$ be a hermitian, holomorphic vector bundle over $X$. Then we get a well-defined bilinear form, which we denote by $\langle,\rangle$, for forms on $X$ with values in $E$ by letting $\langle\alpha\otimes s,\beta\otimes t\rangle:=\alpha\wedge\bar{\beta}\ (s,t)_h$ for forms $\alpha,\beta$ and sections $s,t$, and then extend to arbitrary forms with values in $E$ by linearity. Furthermore we denote the Chern connection associated to this bilinear form by $D=D'+\dbar$ and the curvature by $\Theta=D^2=D'\dbar+\dbar D'$.

Now let $\{dz_j\}$ be orthonormal coordinates at a point and let $\alpha$ be a form of arbitrary bidegree with values in $E$ so that
$$\alpha=\sum\alpha_{IJ}dz_I\wedge d\bar{z}_J$$
where $\{\alpha_{IJ}\}$ are sections of $E$. Then one can show that the norm
\be\label{eq:norm}
\|\alpha\|^2:=\sum\|\alpha_{IJ}\|_h^2
\ee
is independent of the choice of orthonormal basis and that if $\eta$ is an $E-$valued form of bidegree $(p,0)$, then
\be\label{eq:norm_dV}
\|\eta\|^2dV_\omega=c_p\langle\eta,\eta\rangle\wedge\omega_{n-p}
\ee
where $c_p:=i^{p^2}$, $dV_\omega:=\omega^n/n!$ and $\omega_{n-p}:=\omega^{n-p}/(n-p)!$. A similar formula holds for $E-$valued forms of bidegree $(0,q)$, so that $\|\eta\|=\|\bar{\eta}\|$. We denote the set of holomorphic forms of bidegree $(p,0)$ with values in $E$ by $\Omega^{(p,0)}(X,E)$

Polarizing (\ref{eq:norm}) we see that if
$$\beta=\sum\beta_{IJ}dz_I\wedge d\bar{z}_J$$
is another form with values in $E$ which is of the same bidegree as $\alpha$, we get a well defined inner product on $E-$valued forms through
$$(\alpha,\beta):=\sum(\alpha_{IJ},\beta_{IJ})_h.$$
With respect to this inner product we can then define the formal adjoint of the $\dbar$ operator with respect to the metric $h$ through
\be\label{eq:adjoint}
\int_X(\dbar\alpha,\beta)dV_\omega=\int_X(\alpha,\dbar^*_h\beta)dV_\omega
\ee
for all $E-$valued forms $\alpha$ and $\beta$ of appropriate bidegrees.

Given any $(n,p)-$form $\alpha$ it follows from a computation in orthonormal coordinates that there exists an $(n-p,0)-$form $\gamma_\alpha$ such that
$$\alpha=\gamma_\alpha\wedge\omega_p.$$
Namely, if $\alpha$ is given in orthonormal coordinates by
$$\alpha=\sum_{|J|=p}\alpha_Jdz\wedge d\bar{z}_J$$
where $dz:=dz_1\wedge\ldots\wedge dz_n$, then $\gamma_\alpha$ will be given by
$$\gamma_\alpha=\sum_{|J|=p}\varepsilon_J\alpha_Jdz_{J^c}$$
where $\varepsilon_J$ are unimodular constants. It is immediate that the existence of $\gamma_\alpha$ is not affected by requiring $\alpha$ to be $E-$valued and furthermore it is clear that in this case
$$\|\alpha\|^2=\|\gamma_\alpha\|^2.$$
Together with (\ref{eq:norm_dV}) this in turn implies that
$$c_{n-p}\langle\alpha,\gamma_\alpha\rangle=c_{n-p}\langle\gamma_\alpha,\gamma_\alpha\rangle\wedge\omega_p=\|\gamma_\alpha\|^2dV_\omega=\|\alpha\|^2dV_\omega$$
and polarizing this we arrive at
\be\label{eq:product}
(\alpha,\beta)dV_\omega=c_{n-p}\langle\alpha,\gamma_\beta\rangle
\ee
for any other $E-$valued $(n,p)-$form $\beta$.

Using this last formula we can deduce a very useful relation between the formal adjoint of $\dbar$ and the $(1,0)-$part of the Chern connection. If we let $\alpha$ be an $E-$valued $(n,p-1)-$form but keep $\beta$ as before we have that on the one hand
$$\int_X(\dbar\alpha,\beta)dV_\omega = \int_Xc_{n-p}\langle\dbar\alpha,\gamma_\beta\rangle=(-1)^{n-p}\int_Xc_{n-p}\langle\alpha,D'\gamma_\beta\rangle$$
and on the other hand
$$\int_X(\alpha,\dbar_h^*\beta)dV_\omega=\int_Xc_{n-p+1}\langle\alpha,\gamma_{\dbar_h^*\beta}\rangle=\int_X(-1)^{n-p}ic_{n-p}\langle\alpha,\gamma_{\dbar_h^*\beta}\rangle.$$
Hence we see that
$$D'\gamma_\beta=-i\gamma_{\dbar_h^*\beta}$$
so that in particular
\be\label{eq:ad_norm}
\|D'\gamma_\beta\|^2=\|\gamma_{\dbar_h^*\beta}\|^2=\|\dbar_h^*\beta\|^2.
\ee

Finally we end this section by some remarks concerning curvature in the sense of Griffiths and Nakano. Given a hermitian metric $h$ on a holomorphic vector bundle $E$, the curvature $\Theta$ associated to $h$ is locally a matrix of $(1,1)-$forms which we write as
$$\Theta=\sum_{j,k=1}^n\Theta_{jk}dz_j\wedge d\bar{z}_k$$
where $\Theta_{jk}$ are $r\times r$ matrix-values functions on $X$, $r$ being the rank of $E$. We then say that $E$ is non-negatively curved in the sense of Griffiths if for any section $u$ of $E$ and any vector $v$ in $\C^n$
$$\sum_{j,k=1}^n\big(\Theta_{jk}u,u\big)v_j\bar{v}_k\geq0.$$
$E$ is said to be non-negatively curved in the sense of Nakano if
$$\sum_{j,k=1}^n\big(\Theta_{jk}u_j,u_k\big)\geq0$$
for any $n-$tuple $(u_1,\ldots,u_n)$ of sections of $E$. Taking $u_j=uv_j$ we see that non-negativity in the sense of Nakano implies non-negativity in the sense of Griffiths.

Now if $\gamma$ is an $E-$valued $(n-1,0)-$form we can locally write it as
$$\gamma=\sum_{j=1}\gamma^j\widehat{dz_j}$$
where $\widehat{dz_j}$ denotes the wedge product of all $dz_k$ except $dz_j$ ordered so that $dz_j\wedge\widehat{dz_j}=dz_1\wedge\ldots\wedge dz_n$. One can then verify that
\be\label{eq:Nakano}
ic_{n-1}\langle\Theta\wedge\gamma,\gamma\rangle=\sum_{j,k=1}^n\big(\Theta_{jk}\gamma^j,\gamma^k\big)dV_\omega.
\ee
Hence if $\Theta$ is non-negatively curved in the sense of Nakano then
$$ic_{n-1}\langle\Theta\wedge\gamma,\gamma\rangle\geq0$$
for all $E-$valued $(n-1,0)-$forms $\gamma$.

\section{The adjoint version}
\noindent The aim of this section is to prove Theorem \ref{thm:1.1}. Now as mentioned in the introduction we begin by showing that finding an extension is equivalent to solving a $\dbar-$equation.

Let $u\in\Omega^{(n-1,0)}(S,K_S\otimes E|_S)$ and assume first that there exists an extension $U\in\Omega^{(n,0)}(X,K_X\otimes(S)\otimes E)$ such that
$$U=ds\wedge u$$
on $S$. Since $s$ is a holomorphic section vanishing to degree one precisely on $S$, the Lelong-Poincar\'e formula says that $[S]$, the current of integration on $S$, is given by
$$[S]=\frac{i}{2\pi}\partial\dbar\log|s|^2=\frac{i}{2\pi}\partial s\wedge\dbar\left(\frac{1}{s}\right).$$
If we set
$$v'=-\frac{i}{2\pi}\frac{U}{s}$$
then $v'$ is a section of $K_X\otimes E$ and we get that
$$\dbar v'=-\frac{i}{2\pi}\dbar\left(\frac{1}{s}\right)\wedge U=\frac{i}{2\pi}ds\wedge\dbar\left(\frac{1}{s}\right)\wedge u=u\wedge[S]$$
since $\dbar(1/s)$ vanishes outside of $S$. Hence we see that if such an extension exists then $v'$ is a solution of the $\dbar-$equation
\be\label{eq:dbar}
\dbar v=u\wedge[S].
\ee

Conversely, assume now that we can solve (\ref{eq:dbar}). If we extend $u$ smoothly in an arbitrary way to an $E-$valued $(n-1,0)-$form on the whole of $X$ and let
$$v'=-\frac{i}{2\pi}\frac{ds\wedge u}{s},$$
then $\dbar v'=u\wedge[S]$ is independent of the choice of extension since $[S]$ is supported on $S$. Hence $v-v'=h$ is holomorphic which in particular implies that
$$\dbar(sv)=\dbar(sv')=s\dbar v'=su\wedge[S]=0$$ 
so $sv$ is also holomorphic and satisfies
$$2\pi isv=ds\wedge u$$
on $S$. Thus $U=2\pi isv$ solves the extension problem.

To solve (\ref{eq:dbar}) we will use the method of proof for solving the $\dbar-$equation developed by H\"ormander, thereby obtaining an estimate for the solution in addition to existence. In fact if one is satisfied with just knowing that an extension exists and is not interested in the estimate then this will follow from the Nakano vanishing theorem if one also assumes that $\Theta>_N0$. This is a consequence of the well-known fact that the cohomology defined with currents and the cohomology defined with smooth forms are isomorphic, see e.g. \cite{GH}, Chapter 3.1 for a proof.

Now the main ideas behind H\"ormander's method of proof for $\dbar v=f$ where $f$ is a $\dbar-$closed $E$-valued $(n,q)$-form are the following, \cite{H2}. First one formulates the problem dually using a weighted scalarproduct
\be\label{eq:dual}
\int_X(v,\dbar^*_h\alpha)dV_\omega=\int_X(f,\alpha)dV_\omega,
\ee
where $\alpha$ is a smooth $E$-valued $(n,q)$-form. Then one shows that for any $\alpha\in Dom(\dbar_h^*)\cap Ker(\dbar)$
\be\label{eq:main}
\int_X\|\alpha\|^2dV_\omega\leq\int_X\|\dbar^*_h\alpha\|^2dV_\omega.
\ee
This is the most involved part of the proof and to do it one needs curvature assumptions and delicate approximation arguments. However, once (\ref{eq:main}) has been established one has shown that
$$\left|\int_X(f,\alpha)dV_\omega\right|^2\leq\int_X\|f\|^2dV_\omega\int_X\|\dbar^*_h\alpha\|^2dV_\omega$$
for all $\alpha\in Dom(\dbar_h^*)\cap Ker(\dbar)$ and then an argument using the Riesz representation theorem yields that there exists an $L^2-$function $v$ that satisfies (\ref{eq:dual}) and such that
$$\int_X\|v\|^2dV_\omega\leq C\int_X\|f\|^2dV_\omega.$$

Since in our case $f=u\wedge[S]$ is not in $L^2$, the analysis here gets more involved. Let $\alpha$ be a smooth $E-$valued $(n,1)-$form and write $\alpha=\gamma\wedge\omega$ (as there is no risk for confusion we will write $\gamma$ instead of $\gamma_\alpha$). Then by (\ref{eq:product})
$$\int_X(f,\alpha)dV_\omega=\int_Xc_{n-1}\langle f,\gamma\rangle=\int_Sc_{n-1}\langle u,\gamma\rangle$$
and so by the Cauchy-Schwartz inequality
$$\left|\int_X(f,\alpha)dV_\omega\right|^2\leq\int_S c_{n-1}\langle u,u\rangle\int_S c_{n-1}\langle\gamma,\gamma\rangle.$$


Now recall that if we managed to find a solution to (\ref{eq:dbar}) then $sv$ would be the sought for $L^2$ extension of $u$. Hence in that case the method of proof would produce an estimate for
\be\label{eq:est}
\int_Xc_n\langle sv,sv\rangle e^{-\psi}=\int_Xc_n\langle v,v\rangle|s|^2e^{-\psi}.
\ee
This observation leads us to study $L^2-$spaces with the weight $e^w$ where $w=-r\log\big(|s|^2e^{-\psi}\big)$ for $0<r<1$, since working backwards through H\"ormander's method of proof we will later see that in order to reach an estimate for (\ref{eq:est}) using the Riesz representation theorem we need to prove that
\be\label{eq:sc_pr}
\left|\int_X(f,\alpha)dV_\omega\right|^2\leq C\int_Xe^w\|\dbar^*_h\alpha\|^2dV_\omega=C\int_X e^wc_n\langle D'\gamma,D'\gamma\rangle
\ee
for any smooth, $E-$valued $(n,1)-$form $\alpha$.

Lemma \ref{lma:basic} below is the counterpart of (\ref{eq:main}) in H\"ormander's theorem and hence one of the main steps towards this goal. Just like (\ref{eq:main}) it is the most involved part of the proof of the extension theorem.

Now in Raufi \cite{R} we establish (\ref{eq:main}) by adapting the $\partial\dbar-$Bochner-Kodaira method to the setting of Nakano-positive vector bundles. This method was introduced by Siu in \cite{S} for negatively curved vector bundles and adapted by Berndtsson in \cite{B1} to positively curved line bundles. Simply put, in the Nakano-positive case, the key observation is that in order to reach (\ref{eq:main}) one should start by making the calculation
{\setlength\arraycolsep{2pt}
\begin{eqnarray}\label{eq:basic}
ic_{n-1}\partial\dbar\langle\gamma,\gamma\rangle &=& ic_{n-1}\Big(\langle\Theta\wedge\gamma,\gamma\rangle-\langle\dbar D'\gamma,\gamma\rangle+\langle\gamma,\dbar D'\gamma\rangle\Big) + \nonumber\\
& & {}\qquad\qquad\quad+\Big(\|\dbar_h^*\alpha\|^2+\|\dbar\gamma\|^2-\|\dbar\alpha\|^2\Big)dV_\omega.
\end{eqnarray}}
\!and then integrate this expression over $X$; see Berndtsson \cite{B1} or Raufi \cite{R} for more details.

\begin{lma}\label{lma:basic}
In the setting of Theorem \ref{thm:1.1}, let $w=-r\log\big(|s|^2e^{-\psi}\big)$ where $0<r<1$. Then for any smooth, $E$-valued $(n,1)$-form $\alpha$ with $\alpha=\gamma\wedge\omega$
$$\int_Sc_{n-1}\langle\gamma,\gamma\rangle\leq C\Big(\int_X e^wc_n\langle D'\gamma,D'\gamma\rangle+\int_X(w+1)\|\dbar\alpha\|^2dV_\omega\Big).$$
\end{lma}

\begin{proof}
We want to use the $\partial\dbar-$Bochner-Kodaira method but in order to get an integral over $S$ an extra twist is needed. Hence we multiply (\ref{eq:basic}) with $w$ before integrating it over $X$. Using Stokes' theorem on the resulting left hand side we then get
$$\int_Xc_{n-1}wi\partial\dbar\langle\gamma,\gamma\rangle=\int_Xc_{n-1}i\partial\dbar w\wedge\langle\gamma,\gamma\rangle.$$
The idea is that by the Lelong-Poincar\'e formula
$$i\partial\dbar w=ri\partial\dbar\psi-r[S]$$
and so
$$\int_Xc_{n-1}i\partial\dbar w\wedge\langle\gamma,\gamma\rangle=\int_Xc_{n-1}ri\partial\dbar\psi\wedge\langle\gamma,\gamma\rangle-r\int_Sc_{n-1}\langle\gamma,\gamma\rangle.$$
Furthermore from the normalization $|s|^2e^{-\psi}\leq e^{-1/\delta}$ we get
$$w=-r\log\big(|s|^2e^{-\psi}\big)\geq\frac{r}{\delta}>0$$
so that in particular
$$\int_Xw\Big(\|\dbar\gamma\|^2+\|\dbar^*_h\alpha\|^2\Big)dV_\omega\geq0.$$
Combining these two observations and regrouping the terms that result from (\ref{eq:basic}) after integration we thus arrive at
{\setlength\arraycolsep{2pt}
\begin{eqnarray}\label{eq:int}
&&c_{n-1}\Big(\int_Xw\langle i\Theta\wedge\gamma,\gamma\rangle-\int_Xri\partial\dbar\psi\wedge\langle\gamma,\gamma\rangle+r\int_S\langle\gamma,\gamma\rangle\Big)\leq\nonumber\\
&&\qquad\quad\leq ic_{n-1}\Big(\int_Xw\langle\dbar D'\gamma,\gamma\rangle-\int_Xw\langle\gamma,\dbar D'\gamma\rangle\Big) + \int_Xw\|\dbar\alpha\|^2dV_\omega.
\end{eqnarray}}

Now from (\ref{eq:Nakano}) we know that (\ref{eq:curv2}) implies that
$$c_{n-1}\langle i\partial\dbar\psi\wedge\gamma,\gamma\rangle\leq c_{n-1}\frac{1}{\delta}\langle i\Theta\wedge\gamma,\gamma\rangle$$
for any smooth $E-$valued $(n-1,0)-$form $\gamma$. Hence
$$c_{n-1}\Big(\int_Xw\langle i\Theta\wedge\gamma,\gamma\rangle-\int_Xri\partial\dbar\psi\wedge\langle\gamma,\gamma\rangle\Big)\geq\int_X\big(w-\frac{r}{\delta}\big)c_{n-1}\langle i\Theta\wedge\gamma,\gamma\rangle\geq0$$
by (\ref{eq:curv1}) combined with the fact that, as we have already seen, $w\geq r/\delta$. Thus our curvature assumptions yield
$$r\int_Sc_{n-1}\langle\gamma,\gamma\rangle\leq ic_{n-1}\Big(\int_Xw\langle\dbar D'\gamma,\gamma\rangle-\int_Xw\langle\gamma,\dbar D'\gamma\rangle\Big) + \int_Xw\|\dbar\alpha\|^2dV_\omega.$$
By Stokes' theorem and the compatibility of our bilinear form we have
$$\int_Xw\langle\dbar D'\gamma,\gamma\rangle=(-1)^{n-1}\int_Xw\langle D'\gamma,D'\gamma\rangle-\int_X\dbar w\wedge\langle D'\gamma,\gamma\rangle$$
and
$$\int_Xw\langle\gamma,\dbar D'\gamma\rangle=(-1)^n\int_Xw\langle D'\gamma,D'\gamma\rangle+(-1)^n\int_X\partial w\wedge\langle\gamma,D'\gamma\rangle$$
so that
{\setlength\arraycolsep{2pt}
\begin{eqnarray}
&& \int_Xwic_{n-1}\langle\dbar D'\gamma,\gamma\rangle-\int_Xwic_{n-1}\langle\gamma,\dbar D'\gamma\rangle= \nonumber\\
&& \quad=2\int_Xwc_n\langle D'\gamma,D'\gamma\rangle+\int_Xc_n\langle\partial w\wedge\gamma,D'\gamma\rangle+\int_Xc_n\langle D'\gamma,\partial w\wedge\gamma\rangle\nonumber
\end{eqnarray}}
\!\!where we have used that $ic_{n-1}(-1)^{n-1}=c_n$. Since $w\geq0$ implies that $w\leq e^w$ the first term on the right hand side causes no problem and by the Cauchy inequality the last two terms are dominated by
$$\int_Xe^wc_n\langle D'\gamma,D'\gamma\rangle+\int_Xe^{-w}c_n\langle\partial w\wedge\gamma,\partial w\wedge\gamma\rangle.$$
Hence altogether we have that
{\setlength\arraycolsep{2pt}
\begin{eqnarray}\label{eq:lma1}
r\int_Sc_{n-1}\langle\gamma,\gamma\rangle & \leq & 3\int_Xe^wc_n\langle D'\gamma,D'\gamma\rangle+\int_Xw\|\dbar\alpha\|^2dV_\omega+ \nonumber\\
&& \qquad\qquad\qquad+\int_Xe^{-w}c_n\langle\partial w\wedge\gamma,\partial w\wedge\gamma\rangle.
\end{eqnarray}}
\!\!The first two terms are exactly what we want and so only the last term remains. To estimate it we will once again use (\ref{eq:basic}) but this time we will multiply it with the term
$$W=1-e^{-w}$$
before integrating over $X$. The idea is that by choosing $W$ in this way we will get that
$$i\partial\dbar W=\big(ri\partial\dbar\psi-i\partial w\wedge\dbar w\big)e^{-w}.$$
Thus by applying Stokes' theorem to the term that results from the left hand side of (\ref{eq:basic}) and rearranging all the terms just like before we will arrive at
{\setlength\arraycolsep{2pt}
\begin{eqnarray}\label{eq:lma2}
&&c_{n-1}\Big(\int_X\!\!\!W\langle i\Theta\wedge\gamma,\gamma\rangle-\!\!\!\int_X\!\!\!e^{-w}\langle ri\partial\dbar\psi\wedge\gamma,\gamma\rangle\Big)+\!\!\int_X\!\!\!e^{-w}c_n\langle\partial w\wedge\gamma,\partial w\wedge\gamma\rangle\leq\nonumber\\
&& \qquad\qquad\leq ic_{n-1}\Big(\int_XW\langle\dbar D'\gamma,\gamma\rangle-\int_XW\langle\gamma,\dbar D'\gamma\rangle\Big) + \int_XW\|\dbar\alpha\|^2dV_\omega.
\end{eqnarray}}
\!\!Once again we can use the curvature assumptions to get rid of the first two terms since
{\setlength\arraycolsep{2pt}
\begin{eqnarray}
\int_XWc_{n-1}\langle i\Theta\wedge\gamma,\gamma\rangle & - & \int_Xe^{-w}rc_{n-1}\langle i\partial\dbar\psi\wedge\gamma,\gamma\rangle\geq\nonumber\\
&&\qquad\qquad\geq\int_X\big(W-\frac{r}{\delta}e^{-w}\big)c_{n-1}\langle i\Theta\wedge\gamma,\gamma\rangle\nonumber
\end{eqnarray}}
\!\!which is non-negative by (\ref{eq:curv1}) and the fact that
$$W-\frac{r}{\delta}e^{-w}=1-\big(1+\frac{r}{\delta}\big)e^{-w}\geq 1-\big(1+\frac{r}{\delta}\big)e^{-r/\delta}\geq0$$
for small enough $\delta$.

Just as before we now use Stokes' theorem on the first two terms on the right hand side of (\ref{eq:lma2}). This will once again yield
{\setlength\arraycolsep{2pt}
\begin{eqnarray}
&& \int_XWic_{n-1}\langle\dbar D'\gamma,\gamma\rangle-\int_XWic_{n-1}\langle\gamma,\dbar D'\gamma\rangle= \nonumber\\
&& \quad=2\int_XWc_n\langle D'\gamma,D'\gamma\rangle+\int_Xc_n\langle\partial W\wedge\gamma,D'\gamma\rangle+\int_Xc_n\langle D'\gamma,\partial W\wedge\gamma\rangle.\nonumber
\end{eqnarray}}
\!\!As $W\leq1$ and $e^w\geq1$ for small $\delta$ the first term causes no problems. Furthermore since $\partial W=e^{-w}\partial w $ we get by the Cauchy inequality that the last two terms are dominated by
$$C\Big(\int_Xe^{-w}c_n\langle D'\gamma,D'\gamma\rangle+\int_Xe^{-w}c_n\langle\partial w\wedge\gamma,\partial w\wedge\gamma\rangle\Big)$$
for some constant $C$. The second term here is precisely what we want to estimate and so it can be absorbed in the left hand side of (\ref{eq:lma2}). Also as $e^{-w}\leq1$ we have that for small $\delta$
$$\int_Xe^{-w}c_n\langle D'\gamma,D'\gamma\rangle\leq\int_Xe^wc_n\langle D'\gamma,D'\gamma\rangle.$$
Altogether we get
$$\int_Xe^{-w}c_n\langle\partial w\wedge\gamma,\partial w\wedge\gamma\rangle\leq C\Big(\int_Xe^wc_n\langle D'\gamma,D'\gamma\rangle+\int_X\|\dbar\alpha\|^2dV_\omega\Big).$$
Inserting this into our previous estimate (\ref{eq:lma1}) we finally arrive at
$$\int_Sc_{n-1}\langle\gamma,\gamma\rangle\leq C\Big(\int_Xe^wc_n\langle D'\gamma,D'\gamma\rangle+\int_X(w+1)\|\dbar\alpha\|^2dV_\omega\Big)$$
which proves the lemma.
\end{proof}

Recall that we are trying to estimate the scalar product between $f=u\wedge[S]$ and $\alpha$, with the norm of $\|\dbar^*_h\alpha\|^2$ for a smooth, $E-$valued, $(n,1)-$form $\alpha$ as in (\ref{eq:sc_pr}). This will follow from the previous lemma applied to $\dbar-$closed forms, and hence we want to decompose $\alpha=\alpha_1+\alpha_2$ where $\alpha_1$ is $\dbar-$closed and $\alpha_2$ is orthogonal to the space of $\dbar-$closed forms. However since $f$ is not an $L^2$ form the analysis once again gets a little more involved. This is the content of the following lemma.

\begin{lma}\label{lma:lma2}
In the setting of Theorem \ref{thm:1.1}, let $f=u\wedge[S]$ and $w=-r\log\big(|s|^2e^{-\psi}\big)$ where $0<r<1$. Then
$$\left|\int_X(f,\alpha)dV_\omega\right|^2\leq C\int_Sc_{n-1}\langle u,u\rangle\int_Xe^wc_n\langle D'\gamma,D'\gamma\rangle$$
for all smooth, $E-$valued, $(n,1)-$forms $\alpha$, with $\alpha=\gamma\wedge\omega$.
\end{lma}

\begin{proof}
It follows from the Lelong-Poincar\'e formula for $[S]$ and the fact that $u$ is holomorphic, that $f=u\wedge[S]$ is $\dbar-$closed. Hence if we decompose $\alpha=\alpha_1+\alpha_2$ where $\alpha_1$ is $\dbar-$closed and $\alpha_2$ is orthogonal to the space of $\dbar-$closed forms we would like to deduce that
$$\int_X(f,\alpha)dV_\omega=\int_X(f,\alpha^1)dV_\omega.$$
However, since $f$ is a current and not an $L^2$ form, this will require some work. There are two main facts here which we will use but not prove, see e.g. \cite{W}, Chapter 6. The first one is that since $X$ is compact, the ranges of $\dbar$ and $\dbar^*_h$ are closed. Hence since $\alpha_2$ is orthogonal to the space of $\dbar-$closed forms, there exists an $E-$valued $(n,2)-$form $\chi$ such that $\alpha^2=\dbar^*_h\chi$. The second fact that we will need is that due to elliptic regularity, both $\alpha^1$ and $\alpha^2$ are still smooth. This implies that $\chi$ is smooth as well. Thus
$$\int_X(f,\alpha^2)dV_\omega=\int_X(f,\dbar_h^*\chi)dV_\omega=\int_X(\dbar f,\chi)dV_\omega=0.$$
In particular, if we let $\alpha^j=\gamma^j\wedge\omega$ we get
$$\left|\int_X(f,\alpha)dV_\omega\right|^2=\left|\int_X(f,\alpha^1)dV_\omega\right|^2=\left|\int_Sc_{n-1}\langle u,\gamma^1\rangle\right|^2$$
which by the Cauchy-Schwartz inequality is dominated by
$$\int_Sc_{n-1}\langle u,u\rangle\int_Sc_{n-1}\langle\gamma^1,\gamma^1\rangle.$$
Since $\alpha^1$ is $\dbar-$closed, by Lemma \ref{lma:basic} this expression is in turn dominated by
$$C\int_Sc_{n-1}\langle u,u\rangle\int_Xe^wc_n\langle D'\gamma^1,D'\gamma^1\rangle.$$
Finally, as $\alpha_2$ is orthogonal to the space of $\dbar-$closed forms, $\dbar^*_h\alpha^2=0$ which combined with (\ref{eq:ad_norm}) gives $D'\gamma^2=0$. Altogether we hence get that
$$\left|\int_X(f,\alpha)dV_\omega\right|^2\leq C\int_Sc_{n-1}\langle u,u\rangle\int_Xe^wc_n\langle D'\gamma,D'\gamma\rangle.$$
\end{proof}
As discussed in the beginning of this section, finding an extension in the setting of Theorem \ref{thm:1.1} is equivalent to solving the $\dbar-$equation (\ref{eq:dbar}). With Lemma \ref{lma:lma2} at our disposal, we can now proceed to do this in essentially the same way as one proves the standard $L^2-$estimate for $\dbar$, H\"ormander \cite{H2}.

\begin{proof}[Proof of Theorem 1.1]
We set $f=u\wedge[S]$ and formulate (\ref{eq:dbar}) dually using our weighted scalar product by noting that it is equivalent to finding an $E-$valued $(n,0)-$form $v$ such that
$$\int_X(v,\dbar^*_h\alpha)dV_\omega=\int_X(f,\alpha)dV_\omega$$
for all smooth, $E-$valued $(n,1)-$forms $\alpha$. By Lemma \ref{lma:lma2} combined with (\ref{eq:ad_norm})
\be\label{eq:rhs}
\left|\int_X(f,\alpha)dV_\omega\right|^2\leq C\int_Sc_{n-1}\langle u,u\rangle\int_Xe^w\|\dbar_h^*\alpha\|^2dV_\omega.
\ee
Hence if we let
$$V:=\{\dbar^*_h\alpha\ ;\ \alpha\ \ \textrm{smooth, }E\!-\!\textrm{valued }(n,1)\!-\!\textrm{form}\}$$
and define an anti-linear functional $L$ on $V$ through
$$L(\dbar^*_h\alpha):=\int_X(f,\alpha)dV_\omega$$
then (\ref{eq:rhs}) says that $L$ is well-defined and of norm not exceeding
$$C\int_Sc_{n-1}\langle u,u\rangle$$
on $V$. By the Hahn-Banach theorem, $L$ can be extended to an anti-linear functional on
$$L_{(n,1)}^2(e^w):=\{\alpha\ \ E\!-\!\textrm{valued }(n,1)\!-\!\textrm{form such that }\int_Xe^w\|\alpha\|^2<\infty \}$$
with the same norm.

Now by the Riesz representation theorem there exists an $\eta\in L_{(n,1)}^2(e^w)$ such that
$$L(\beta)=\int_Xe^w(\eta,\beta)dV_\omega$$
for all $\beta\in L_{(n,1)}^2(e^w)$, and
$$\int_Xe^w\|\eta\|^2dV_\omega\leq C\int_Sc_{n-1}\langle u,u\rangle.$$
In particular if $\beta=\dbar^*_h\alpha$ we get that
$$\int_Xe^w(\eta,\dbar^*_h\alpha)dV_\omega=\int_X(f,\alpha)dV_\omega.$$
Furthermore, since we have chosen $0<r<1$ in $w=-r\log\big(|s|^2e^{-\psi}\big)$, $e^w$ is integrable and so $\dbar^*_h\alpha\in L_{(n,1)}^2(e^w)$ for all smooth, $E-$valued $(n,1)-$forms $\alpha$. Thus $v:=e^w\eta$ is the sought for solution of (\ref{eq:dbar}) and we have that
$$\int_Xe^w\|e^{-w}v\|^2dV_\omega\leq C\int_Sc_{n-1}\langle u,u\rangle.$$
However the left hand side here is nothing but
$$\int_Xe^w\|e^{-w}v\|^2dV_\omega=\int_Xe^{-w}c_n\langle v,v\rangle=\int_Xc_n\langle v,v\rangle\big(|s|^2e^{-\psi}\big)^r,$$
and so by our normalization $|s|^2e^{-\psi}\leq e^{-1/\delta}$
\be\label{eq:str_ineq}
\int_Xc_n\langle sv,sv\rangle e^{-\psi}\leq\int c_n\langle v,v\rangle\big(|s|^2e^{-\psi}\big)^r\leq C\int_Sc_{n-1}\langle u,u\rangle
\ee
for small enough $\delta$.

Hence by the discussion in the beginning of this section $U=2\pi isv$ solves our extension problem and we have have the estimate
$$\int_Xc_n\langle U,U\rangle e^{-\psi}\leq C\int_Sc_{n-1}\langle u,u\rangle.$$
\end{proof}

\begin{remark}
A more careful study of the proof of Theorem \ref{thm:1.1} reveals that in fact a slightly stronger inequality than the one stated holds. Namely as $U=2\pi isv$ it follows from (\ref{eq:str_ineq}) that
$$\int_X\frac{c_n\langle U,U\rangle}{|s|^{2-2r}}e^{-\psi}\leq C\int_Sc_{n-1}\langle u,u\rangle$$
for $0<r<1$. Thus although $v$ is not $L^2$ valued, $v^{1-\varepsilon}$ belongs to $L^2$ for all $\varepsilon>0$.

The same type of improvement can also be made in the non-adjoint setting below.  
\end{remark}

\section{The non-adjoint version and reduction to line bundles}
\noindent As the non-adjoint version of the extension theorem does not involve any forms we will first need to express the norms
$$\int_Xc_n\langle U,U\rangle e^{-\psi}$$
and
$$\int_Sc_{n-1}\langle u,u\rangle$$
in terms of $L^2$-norms with respect to the volume elements on $X$ and $S$ respectively. Let $\eta$ be an $E-$valued $(n,0)-$form. In section 2 we defined the norm of $\eta$ through
$$\|\eta\|^2dV_\omega:=c_n\langle \eta,\eta\rangle$$
and we also noted that if at a point $\eta=\eta_Edw_1\wedge\ldots\wedge dw_n$ where $\{dw_k\}$ denotes an orthonormal basis, then
$$\|\eta\|^2=(\eta_E,\eta_E)_h.$$
However if we have a local basis $\{dz_k\}$ which is not orthonormal we can use the K\"ahler metric to obtain a metric, $e^{-\phi_\omega}$, on the canonical bundle of $X$ through
$$\|\eta\|^2dV_\omega=(\eta_E,\eta_E)_he^{-\phi_\omega}dV_\omega.$$
One way to see this is to observe that locally 
$$c_n\langle\eta,\eta\rangle=(\eta^i_E,\eta^i_E)\frac{c_ndz^i\wedge d\bar{z}^i}{dV_\omega}dV_\omega=:(\eta^i_E,\eta^i_E)e^{-\phi^i_\omega}dV_\omega$$
where $dz^i:=dz^i_1\wedge\ldots\wedge dz_n^i$ and $\eta$ is locally represented as $\eta^i_Edz^i$. Hence the metric locally is defined as
$$\phi^i_\omega:=-\log\left(\frac{c_ndz^i\wedge d\bar{z}^i}{dV_\omega}\right).$$
One can then readily check that if $\{dz_k^j\}$ is another basis with $g_{ij}dz^i=dz^j$ then $\phi^i_\omega-\phi^j_\omega=\log|g_{ij}|^2$ and hence $e^{-\phi_\omega}$ is well-defined as a metric on $K_X$.
Thus
$$\int_Xc_n\langle U,U\rangle e^{-\psi}=\int_X(U,U)_he^{-\phi_\omega-\psi}dV_\omega.$$

To express the integral over $S$ in a similar way we will use $dS$, the volume (or surface) measure on the hypersurface $S$. This is a measure on $X$ and so can be regarded as a current of bidegree $(n,n)$ which in turn induces a current of bidegree $(0,0)$ that we will denote by $\ast dS$, through
$$(\ast dS)dV_\omega=dS.$$
One can then show (see e.g. Berndtsson \cite{B1}) that the current of integration on $S$ can be represented as
$$[S]=\frac{ids\wedge d\bar{s}}{\|ds\|^2}\ast dS$$
where the right hand is defined by taking any local representative of the section $s$, and the norm of $ds$ is defined as usual through
$$\|ds\|^2dV_\omega:=ids\wedge d\bar{s}\wedge\omega_{n-1}.$$
If we now let $U_0$ denote the restriction of $U$ to $S$ so that $U_0=u\wedge ds$ we then have that
$$\int_Sc_{n-1}\langle u,u\rangle=\int_Xc_{n-1}\langle u,u\rangle\wedge\frac{ids\wedge d\bar{s}}{\|ds\|^2}\ast dS=\int_Xc_n\langle U_0,U_0\rangle\frac{\ast dS}{\|ds\|^2}.$$
However from the previous discussion we know that
$$c_n\langle U_0,U_0\rangle=(U_0,U_0)_he^{-\phi_\omega}dV_\omega$$
so that in fact
$$\int_Sc_{n-1}\langle u,u\rangle=\int_Sc_n\langle U_0,U_0\rangle\frac{\ast dS}{\|ds\|^2}=\int_S(U_0,U_0)_he^{-\phi_\omega-\psi}\frac{dS}{\|ds\|^2e^{-\psi}}.$$

We can now proceed to prove the non-adjoint version of the extension theorem.
 
\begin{proof}[Proof of Theorem 1.2]
We can recast the theorem to the setting of Theorem \ref{thm:1.1} by letting $E:=F\otimes(S)^{-1}\otimes K_X^{-1}$ and defining $\tilde{h}:=h\otimes e^{\phi_\omega+\psi}$ as a metric on $E$. The curvature assumptions (\ref{eq:curv1}) and (\ref{eq:curv2}) for $\tilde{h}$ are then equivalent to the assumptions
$$i\Theta^h-(1+\delta)i\partial\dbar\psi\otimes I-i\partial\dbar\phi_{\omega}\otimes I\geq_{N}0$$
and
$$i\Theta^h-i\partial\dbar(\psi+\phi_{\omega})\otimes I\geq_{N}0.$$
Thus Theorem \ref{thm:1.1} and our previous discussion yield that any holomorphic section $U_0$ of $F$ over $S$ extends holomorphically to a section $U$ of $F$ over $X$ with the estimate
$$\int_X(U,U)_hdV_\omega\leq C\int_S(U_0,U_0)_h\frac{dS}{\|ds\|^2e^{-\psi}}.$$
This proves the theorem when $\phi_{K_X}=\phi_\omega$.

For an arbitrary metric $\phi_{K_X}$ on the canonical bundle of $X$ we know that there exists a smooth function $\chi$ such that
$$\phi_{K_X}+\chi=\phi_\omega.$$
Hence we get that
$$\phi_\omega+\psi=(\phi_{K_X}+\chi)+\psi=\phi_{K_X}+(\chi+\psi)$$
so changing $\phi_\omega$ to $\phi_{K_X}$ is equivalent to changing $\psi$ to $\psi+\chi$, and in this setting (\ref{eq:curv1}) and (\ref{eq:curv2}) are equivalent to (\ref{eq:vb_curv1}) and (\ref{eq:vb_curv2}).
\end{proof}

We now turn to the alternative route to obtaining extensions by reformulating the problem in a setting with line bundles. This involves the study of the so called Serre line bundle over the projectivization of a vector bundle, which can be constructed in an abstract, global way as e.g. in \cite{L}. However here we will content ourselves with a, from our perspective, more concrete local description.

Hence let $F$ denote an arbitrary complex vector bundle over $X$ with rank$_{\mathbb{C}}F=r$, and let $F^*$ denote the dual bundle. We can now define a fiber bundle $\pi:\mathbb{P}(F)\to X$ by defining each fiber through $\mathbb{P}(F)_x:=\mathbb{P}(F^*_x)$, the projectivization of an $r-$dimensional vector space. Locally, for an open set $U\subset X$ we have that $F^*\big|_U\simeq U\times\C^{r}$ and then $\mathbb{P}(F)\big|_U\simeq U\times\mathbb{P}^{r-1}$. Furthermore the pullback bundle $\pi^*F^*\to\mathbb{P}(F)$ will then locally be given by
$$\pi^*F^*\big|_U\simeq U\times\mathbb{P}^{r-1}\times\C^{r}$$
and so we can define the tautological line subbundle $\mathcal{O}_{\mathbb{P}(F)}(-1)$ of $\pi^*F^*$ as
$$\mathcal{O}_{\mathbb{P}(F)}(-1)\big|_U:=\{(x,[w],z)\ ;\ z\in[w]\}.$$
The line bundle $\mathcal{O}_{\mathbb{P}(F)}(1)$ (which we, following Lazarsfeld \cite{L}, have called the Serre line bundle) is then defined as the dual of $\mathcal{O}_{\mathbb{P}(F)}(-1)$. The notation is justified by the fact that fiberwise this is nothing but the usual line bundle $\mathcal{O}(1)$ over $\mathbb{P}^r$. Thus we have that the global holomorphic sections of $\mathcal{O}_{\mathbb{P}(F)}(1)$ over any fiber are in one-to-one correspondence with the linear forms on $F^*_x$, i.e. with the elements of $F_x$; this is the reason for projectivizing $F^*$ instead of $F$.

Now in the setting of Theorem \ref{thm:1.2} instead of proving the existence of extensions directly as we have done one may reduce the extension problem by studying the line bundles $\mathcal{O}_{\mathbb{P}(F)}(1)$ and $\pi^*(S)$ over $\mathbb{P}(F)$ instead of $F$ and $(S)$ over $X$. Here the usual extension theorem of Ohsawa-Takegoshi applies which states that for such an extension to exist $\mathcal{O}_{\mathbb{P}(F)}(1)$ and $\pi^*(S)$ must have metrics that satisfy the curvature assumptions (\ref{eq:lb_curv1}) and (\ref{eq:lb_curv2}). In our setting $\mathcal{O}_{\mathbb{P}(F)}(1)$ and $\pi^*(S)$ inherit the metrics $\pi^*h$ and $\pi^*\psi$ from $F$ and $(S)$, and we are interested in finding out what conditions on $h$ and $\psi$ that arise from (\ref{eq:lb_curv1}) and (\ref{eq:lb_curv2}). For this we will utilize two well-known facts concerning $\mathcal{O}_{\mathbb{P}(F)}(1)\to\mathbb{P}(F)$ and its relation to $F\to X$ that we will not prove.

The first of these is that if the metric $h$ on $F$ is non-negatively curved in the sense of Griffiths, then the inherited metric $\pi^*h$ on $\mathcal{O}_{\mathbb{P}(F)}(1)$ is also non-negative. Following Hartshorne \cite{H1} we will call a vector bundle $F$ ample if it has the property that $\mathcal{O}_{\mathbb{P}(F)}(1)$ can be equipped with a non-negatively curved metric. It is a well-known conjecture of Griffiths that the converse to the previous statement holds in the sense that if $F$ is ample then one can also find a metric on $F$ that is non-negative in the sense of Griffiths, \cite{G}. However, if the metric on $\mathcal{O}_{\mathbb{P}(F)}(1)$ stems from a metric on $F$, so that one does not need to construct a metric on $F$, then one can show that the non-negativity of $\mathcal{O}_{\mathbb{P}(F)}(1)$ is in fact equivalent with the non-negativity of $F$ in the sense of Griffiths.

The second property that we will need is that for any line bundle $L$ on $X$, $\P(F)$ is isomorphic to $\P(F\otimes L)$ via an isomorphism under which $\mathcal{O}_{\mathbb{P}(F\otimes L)}(1)$ is isomorphic to $\mathcal{O}_{\mathbb{P}(F)}(1)\otimes\pi^*L$, see e.g. Lazarsfeld \cite{L}.

We now proceed to show that the curvature assumptions (\ref{eq:lb_curv1}) and (\ref{eq:lb_curv2}) in fact imply those in Theorem \ref{thm:1.2}. We formulate this as a proposition.

\begin{prop}
If there exists a metric on $K_{\P(F)}$ which together with the inherited metrics $\pi^*h$ and $\pi^*\psi$ on $\mathcal{O}_{\mathbb{P}(F)}(1)$ and $\pi^*(S)$ satisfies the curvature assumptions (\ref{eq:lb_curv1}) and (\ref{eq:lb_curv2}), then there exists a metric on $K_X$ that along with the metrics $h$ and $\psi$ satisfies the curvature assumptions (\ref{eq:vb_curv1}) and (\ref{eq:vb_curv2}).
\end{prop}

\begin{proof}
We will content ourselves with proving that (\ref{eq:lb_curv2}) implies (\ref{eq:vb_curv2}). That (\ref{eq:lb_curv1}) implies (\ref{eq:vb_curv1}) can then be proved in exactly the same way.

Now we start by rewriting (\ref{eq:lb_curv2}) using the well-known fact that
$$K_{\P(F)}=\mathcal{O}_{\mathbb{P}(F)}(-r)\otimes\pi^*\big(\det F\otimes K_X\big),$$
see e.g. Griffiths \cite{G}. Here $\mathcal{O}_{\mathbb{P}(F)}(-r)$ denotes the tensor product of $\mathcal{O}_{\mathbb{P}(F)}(-1)$ $r$ times and in what follows we will use this notation repeatedly. Hence we get that (\ref{eq:lb_curv2}) is equivalent to
$$\mathcal{O}_{\mathbb{P}(F)}(r+1)\otimes\pi^*(S)^{-1}\otimes\pi^*(\det F\otimes K_X)^{-1}\geq0$$
by which we mean that this line bundle has a metric which is non-negatively curved. Taking the $(r+1)$:th root, i.e. multiplying the curvature tensor by $1/(r+1)$, we get that
$$\mathcal{O}_{\mathbb{P}(F)}(1)\otimes\pi^*(S)^{-\frac{1}{r+1}}\otimes\pi^*(\det F\otimes K_X)^{-\frac{1}{r+1}}\geq0,$$
which by the above mentioned isomorphism is equivalent to
$$\mathcal{O}_{\mathbb{P}(F\otimes((S)\otimes\det F\otimes K_X)^{-1/(r+1)})}\geq0.$$
Thus by our previous discussion, as we know that the metric on $\mathcal{O}_{\mathbb{P}(F)}(1)$ comes from a metric on $F$, this implies that
$$F\otimes\big((S)\otimes\det F\otimes K_X\big)^{-\frac{1}{r+1}}\geq_G0.$$

Now a fundamental theorem of Demailly and Skoda \cite{DS} states that if a vector bundle $E$ is non-negatively curved in the sense of Griffiths then $E\otimes\det E$ is non-negatively curved in the sense of Nakano. Utilizing this we get that
$$F\otimes\big((S)\otimes\det F\otimes K_X\big)^{-\frac{1}{r+1}}\otimes\det F\otimes\big((S)\otimes\det F\otimes K_X\big)^{-\frac{r}{r+1}}\geq_N0.$$
But this is nothing but
$$F\otimes(S)^{-1}\otimes K_X^{-1}\geq_N0,$$
which yields (\ref{eq:vb_curv2}).
\end{proof}

\begin{remark}
One might (rightfully) object that the line bundle assumptions (\ref{eq:lb_curv1}) and (\ref{eq:lb_curv2}) yield the existence of an extension even if the metrics involved do not stem from metrics on the vector bundle, which is a key assumption for our argument. However by a theorem of Berndtsson \cite{B2} if $F$ is ample then $F\otimes\det F$ is non-negative in the sense of Nakano so (\ref{eq:lb_curv1}) and (\ref{eq:lb_curv2}) imply (\ref{eq:vb_curv1}) and (\ref{eq:vb_curv2}) even if the metric on $\mathcal{O}_{\mathbb{P}(F)}(1)$ is not inherited from a metric on $F$.

\end{remark}


\end{document}